\documentclass[a4paper,11pt,leqno]{article}
\usepackage{tikz-cd}
\usepackage{amsmath,amssymb,amsthm,mathrsfs,enumitem}
\usepackage{xcolor}
\usepackage[left=25mm,right=25mm,top=30mm,bottom=30mm]{geometry}
\usepackage{cleveref}
\usepackage{graphicx}

\title{Relative stable equivalences of Morita type
for the principal blocks of finite groups
and relative Brauer indecomposability}
\author{Naoko Kunugi and Kyoichi Suzuki*}
\date{}

\crefname{definition}{Definition}{Definitions}
\crefname{theorem}{Theorem}{Theorems}
\crefname{lemma}{Lemma}{Lemmas}
\crefname{proposition}{Proposition}{Theorems}
\crefname{corollary}{Corollary}{Corollaries}
\crefname{remark}{Remark}{Remarks}
\crefname{table}{Table}{Tables}
\crefname{section}{Section}{Sections}

\newtheorem{definition}{Definition}[section]
\newtheorem{theorem}[definition]{Theorem}
\newtheorem{lemma}[definition]{Lemma}
\newtheorem{proposition}[definition]{Proposition}

\newtheorem{remark}[definition]{Remark}

\begin{document}
\maketitle

    \noindent{Abstract.}
    We discuss representations of finite groups
    having a common central $p$-subgroup $Z$,
    where $p$ is a prime number.
    For the principal $p$-blocks,
    we give a method of constructing
    a relative $Z$-stable equivalence of Morita type,
    which is a generalization of a stable equivalence of Morita type,
    and was introduced by Wang and Zhang in a more general setting.
    Then we generalize Linckelmann's results on stable equivalences of Morita type
    to relative $Z$-stable equivalences of Morita type.
    We also introduce the notion of relative Brauer indecomposability,
    which is a generalization of the notion of Brauer indecomposability.
    We give an equivalent condition for Scott modules
    to be relatively Brauer indecomposable,
    which is an analogue of that given by Ishioka and the first author.

    \section{Introduction}

    Morita equivalences for the principal $p$-blocks of finite groups
    has been constructed by using results
    due to Brou\'{e} \cite{broue-1994} and Linckelmann \cite{Linckelmann-1996},
    where $p$ is a prime.
    Brou\'{e} \cite{broue-1994} introduced the notion
    of stable equivalences of Morita type,
    and developed a method of constructing those
    for principal blocks.
    Linckelmann \cite{Linckelmann-1996} gave an equivalent condition 
    for stable equivalences of Morita type
    between indecomposable selfinjective algebras
    to be in fact Morita equivalences.
    Thus, we may construct a stable equivalence of Morita type
    by using Brou\'{e}'s method and lift it to a Morita equivalence
    by using Licnkelmann's result.
    In this way, Morita equivalences has been constructed in some cases
    (see for example \cite{Okuyama-preprint-1997}
    and \cite{Koshitani-Lassueur-2020}).

    However we cannot use Brou\'{e}'s method
    in the case where $G$ and $G'$ have a common nontrivial central $p$-subgroup.
    If $G$ and $G'$ are finite groups with a common Sylow $p$-subgroup $P$
    and the same fusion system on $P$,
    then Brou\'{e}'s method constructs a stable equivalence of Morita type
    between the principal blocks of $G$ and $G'$
    by gluing Morita equivalences between the principal blocks of
    the centralizers of the nontrivial subgroups of $P$.
    Hence we cannot use the method
    if $G$ and $G'$ have a common nontrivial central $p$-subgroup
    because its centralizers are $G$ and $G'$ themselves.

    In \cite{Koshitani-Lassueur-2020-2},
    it was claimed to have shown that
    for a common central $p$-subgroup $Z$ of $G$ and $G'$,
    the principal blocks of $G/Z$ and $G'/Z$ are Morita equivalent
    if and only if so are the principal blocks of $G$ and $G'$
    (see \cite[Lemma 3.3 (b)]{Koshitani-Lassueur-2020-2}).
    Hence we may construct a Morita equivalence between
    the principal blocks of $G$ and $G'$
    since we may construct that between
    the principal blocks of $G/Z$ and $G'/Z$
    by using the results due to Brou\'{e} and Linckelmann.
    However it seems that the proof of \cite[Lemma 3.1 (b)]{Koshitani-Lassueur-2020-2},
    which is needed for the proof of \cite[Lemma 3.3 (b)]{Koshitani-Lassueur-2020-2},
    is not sufficient.

    Therefore we give a method of constructing Morita equivalences
    for the principal blocks of finite groups
    with a common nontrivial central $p$-subgroup
    by generalizing the results due to Brou\'{e} and Linckelmann
    to relative stable equivalences of Morita type.
    The notion of relative stable equivalences of Morita type
    was introduced by Wang and Zhang \cite{Wang-Zhang-2018},
    and is a generalization of stable equivalences of Morita type.
    We use a subgroup version of this notion (see \cref{sec-relative-stm}).
    If $G$ and $G'$ have a common nontrivial central $p$-subgroup $Z$,
    we give a method of constructing
    a relative $Z$-stable equivalence of Morita type between
    the principal blocks of $G$ and $G'$:
    
    \begin{theorem}\label{generalization-Broue-gluing}
        Let $k$ be an algebraically closed filed of characteristic $p>0$.
        Let $G$ and $G'$ be finite groups with a common Sylow $p$-subgroup $P$
        such that $G$ and $G'$ have the same fusion system on $P$,
        and $M=S(G\times G',\Delta P)$,
        the Scott $k[G\times G']$-module with vertex $\Delta P$.
        Assume that $Z$ is a subgroup of $P$ central in $G$ and $G'$.
        Then the following are equivalent:
        \begin{enumerate}[label={$(\mathrm{\roman*})$}]
            \item The pair $(M(\Delta Q),{M(\Delta Q)}^{*})$
            of the Brauer construction of $M$
            with respect to $\Delta Q$ and its dual
            induces a Morita equivalence 
            between the principal blocks of $kC_{G}(Q)$ and $kC_{G'}(Q)$
            for any subgroup $Q$ of $P$ properly containing $Z$.
            \item The pair $(M,M^{*})$ induces a relative
            $Z$-stable equivalence of Morita type
            between the principal blocks of $kG$ and $kG'$.
        \end{enumerate}
    \end{theorem}

    Then we generalize Linckelmann's result \cite[Theorem 2.1]{Linckelmann-1996} for
    a relative $Z$-stable equivalence of Morita type:
    \begin{theorem}\label{generalization-Linckelmann}
        Let $k$ be an algebraically closed field of characteristic $p>0$.
        Let $G$ and $G'$ be finite groups,
        and $B$ and $B'$ blocks of $kG$ and $kG'$, respectively,
        with a common nontrivial defect group $P$
        such that $G$ and $G'$ have the same fusion system on $P$.
        Let $M$ be a $B$-$B'$-bimodule that is a
        $\Delta P$-projective $p$-permutation $k[G\times G']$-module.
        Assume that for a subgroup $Q$ of $P$,
        the pair $(M,M^{*})$ induces
        a relative $Q$-stable equivalence of Morita type
        between $B$ and $B'$.
        Then the following hold:
        \begin{enumerate}[label={$(\mathrm{\roman*})$}]
            \item Up to isomorphism,
            $M$ has a unique indecomposable summand
            that is non $Q\times Q$-projective,
            considered as a $k[G\times G']$-module.
        \end{enumerate}

        Moreover, assume that $Z$ is a proper subgroup of $P$
        such that $Z$ is central in $G$ and $G'$.
        If $Q=Z$, then the following hold:
        \begin{enumerate}[label={$(\mathrm{\roman*})$}]
            \setcounter{enumi}{1}
            \item If $M$ is a trivial source module with vertex $\Delta P$,
                then for any simple $B$-module $S$,
                the $B'$-module $S\otimes_{B}M$ is indecomposable,
                and non $Z$-projective, considered as a $kG'$-module.
            
            \item The pair $(M,M^{*})$ induces a Morita equivalence
            between $B$ and $B'$ if and only if
            for any simple $B$-module $S$,
            the $B'$-module $S\otimes_{B}M$ is simple.
        \end{enumerate}
        
    \end{theorem}
    
    We also introduce the notion of relative Brauer indecomposability
    (see \cref{def-relative-brauer-indecomposable})
    and give an equivalent condition
    for Scott modules to be relatively Brauer indecomposable.
    The notion of Brauer indecomposability was introduced in
    \cite{Kessar-Kunugi-Mitsuhashi-2011}.
    The Brauer indecomposability of Scott modules
    plays an important role in Brou\'{e}'s method.
    Ishioka and the first author gave an equivalent condition
    for Scott modules to be Brauer indecomposable
    (see \cite[Theorem 1.3]{Ishioka-Kunugi-2017}).
    Although Brauer indecomposability of Scott modules is also useful
    for \cref{generalization-Broue-gluing},
    somewhat more general condition is more appropriate.
    Therefore we introduce the notion of relative Brauer indecomposability
    and generalize the result \cite[Theorem 1.3]{Ishioka-Kunugi-2017}
    to this notion:

    \begin{theorem}\label{generalization-Ishioka-Kunugi}
        Let $G$ be a finite group, $P$ a $p$-subgroup of $G$, and $M=S(G,P)$.
        Suppose that the fusion system $\mathcal{F}_{P}(G)$ is saturated,
        and $R$ is a subgroup of $P$.
        Then the following are equivalent:
        \begin{enumerate}[label={$(\mathrm{\roman*})$}]
            \item The module $M$ is relatively $R$-Brauer indecomposable.
            \item The module 
            $S(N_{G}(Q),N_{P}(Q)){\downarrow}_{QC_{G}(Q)}^{N_{G}(Q)}$
            is indecomposable for each fully normalized subgroup $Q$
            of $P$ containing a $G$-conjugacy of $R$.
        \end{enumerate}
        Moreover, if these conditions hold,
        then $M(Q)\cong S(N_{G}(Q),N_{P}(Q))$
        for any fully normalized subgroup $Q$ of $P$
        containing a $G$-conjugacy of $R$.
    \end{theorem}

    This paper is organized as follows:
    in Section 2, we establish some notation and facts
    used throughout the paper.
    We also recall the definitions and some facts on fusion systems.
    In Section 3, we recall the definition of
    a relative stable equivalence of Morita type,
    and investigate its properties under suitable hypotheses.
    In Section 4, we prove 
    \cref{generalization-Broue-gluing} and \cref{generalization-Linckelmann}.
    In Section 5, we define relative Brauer indecomposability,
    and prove \cref{generalization-Ishioka-Kunugi}.
    In Section 6,
    we give an example of constructing a Morita equivalence
    between the principal blocks of $SL_{2}(3)$ and $SL_{2}(11)$
    in characteristic $2$.

    \section{Notation and preliminaries}
    Throughout this paper, we assume that
    $k$ is an algebraically closed field of characteristic $p>0$,
    $G$ is a finite group,
    and modules are finitely generated right modules, unless otherwise stated.
    
    We write $H\leq G$ if $H$ is a subgroup of $G$,
    and write $H\unlhd G$ if $H$ is a normal subgroup of $G$.
    For subgroups $H$ and $K$ of $G$,
    we write $H\leq_{G}K$
    if $H$ is conjugate in $G$ with a subgroup of $K$.
    In particular, if $H$ is a proper subgroup of $G$,
    then we write $H<G$ for $H\leq G$,
    and $H\vartriangleleft G$ for $H\unlhd G$.
    We also write $H<_{G}K$ if $H$ is conjugate in $G$
    with a proper subgroup of $K$.
    We set $H^{g}=g^{-1}Hg$,
    and write $[H\backslash G]$ for a set of representatives
    of the right cosets of $H$ in $G$.
    We also write $[H\backslash G/K]$ for a set of representatives
    of the double cosets of $H$ and $K$ in $G$.
    We write $Z(G)$ for the center of $G$.
    We write $\Delta G=\{(g,g)\mid g\in G\}\leq G\times G$.
    
    Let $H$ be a subgroup of $G$.
    For a $kG$-module $M$,
    we write $M{\downarrow}_{H}^{G}$ for the restriction of $M$ to $H$.
    For a $kH$-module $N$, we write $N{\uparrow}_{H}^{G}$
    for the induced $kG$-module of $N$.
    We write $k_{G}$ for the trivial $kG$-module,
    and $B_{0}(G)$ for the principal block of $kG$.
    We write $J(kG)$ for the Jacobson radical of $kG$.
    For modules $U$ and $V$,
    we write $U\otimes V$ for $U\otimes_{k}V$,
    and $V^{*}=\mathrm{Hom}_{k}(V,k)$ for the $k$-dual of $V$.
    If $U$ is a left module, and $V$ is a right module,
    then we consider $U\otimes V$ as a bimodule,
    and $V^{*}$ as a left module,
    unless otherwise stated.

    For a subgroup $H$ of $G$,
    there is a unique indecomposable summand of $k_{H}{{\uparrow}^{G}}$
    having $k_{G}$ as a direct summand of the top.
    This indecomposable summand is called the Scott module with respect to $H$,
    and denoted by $S(G,H)$.
    We use the fact that if $H$ and $H'$ are subgroups of $G$,
    and $Q$ and $Q'$ are Sylow $p$-subgroups of $H$ and $H'$, respectively,
    then $S(G,H)\cong S(G,H')$ if and only if $Q$ and $Q'$ are conjugate in $G$.
    (see \cite[Chapter 4, Corollary 8.5]{Nagao-Tsushima-book-1989}).
    In particular, it follows that $S(G,H)\cong S(G,Q)$,
    and $S(G,Q)$ has $Q$ as a vertex.
    Therefore we refer to $S(G,Q)$ as the Scott module with vertex $Q$.
    A $kG$-module is called a $p$-permutation module if
    it is a direct summand of $\bigoplus_{i=1}^{r}k_{H_i}{\uparrow}^{G}$
    for some subgroups $H_{i}$ of $G$.
    An indecomposable $p$-permutation module is called a trivial source module.
    The Scott module $S(G,H)$ is a trivial source module.
    
    We recall the definition of the Brauer construction
    and its basic facts.
    For a $kG$-module $M$ and a $p$-subgroup $Q$ of $G$,
    the Brauer construction $M(Q)$ of $M$ with respect to $Q$
    is the $kN_{G}(Q)$-module defined as follows:
    \[
        M(Q)=M^{Q}/\sum_{R<Q}\mathrm{tr}_{R}^{Q}(M^{R}),
    \]
    where $M^{Q}$ is the set of fixed points of $Q$ in $M$,
    and $\mathrm{tr}_{R}^{Q}:M^{R}\rightarrow M^{Q}$
    is a linear map given by
    $\mathrm{tr}_{R}^{Q}(m)=\sum_{t\in [R\backslash Q]}mt$.

    \begin{lemma}\textnormal{(see \cite[(1.3)]{broue-1985},
        \cite[{Corollary} 27.7]{Thevenaz-book-1995},
        and also \cite[\text{Proposition }5.10.3]{Linckelmann-book-2018-vol1})}
        \label{Brauer-construction-vertex}
        Let $M$ be an indecomposable $kG$-module and $Q$ a $p$-subgroup of $G$.
        Then the following hold:
        \begin{enumerate}[label={$(\mathrm{\roman*})$}]
            \item If $M(Q)\neq 0$, then $Q$ is contained in a vertex of $M$.
            \item In particular, if $M$ is a trivial source module,
            then $M(Q)\neq 0$ if and only if $Q$ is contained in a vertex of $M$.
        \end{enumerate}
    \end{lemma}

    \begin{lemma}\label{Brauer-construction-trivial}
        Let $Z$ be a $p$-subgroup of $Z(G)$.
        If $M$ is a trivial source $kG$-module
        with vertex containing $Z$, then $M(Z)=M$.
    \end{lemma}
    \begin{proof}
        If $Z=1$, then the assertion clearly holds
        by the definition of the Brauer construction,
        and hence we assume that $Z\neq 1$.
        Since $M$ is a trivial source module with vertex containing $Z$,
        it is a direct summand of $k_{H}{\uparrow}^{G}$
        for some subgroup $H$ of $G$ containing $Z$.
        Since $Z\leq Z(G)$,
        it follows that $Z$ acts trivially on $k_{H}{\uparrow}^{G}$,
        and in particular acts trivially on $M$.
        Hence for any subgroup $R$ of $Z$, we have that $M^{R}=M$.
        For any proper subgroup $R$ of $Z$
        and $m\in M=M^{R}$, we have that
        $\textrm{tr}_{R}^{Z}(m)=\sum_{t\in [R\backslash Z]}mt=|Z:R|m=0$,
        and the result follows.
    \end{proof}
    
    We recall some definitions of fusion systems.
    However, it may suffice to know the facts
    in \cref{important-properties-fusion-system} below.
    For subgroups $H$ and $K$ of $G$,
    we write
    \[
        \mathrm{Hom}_{G}(H,K)=\{\varphi\in\mathrm{Hom}(H,K)\mid\varphi=c_{g}
        \text{ for some $g\in G$ such that $H^{g}\leq K$}\},
    \]
    where $c_g$ is a conjugation map.
    Let $P$ be a $p$-subgroup of $G$.
    The \textit{fusion system} of $G$ over $P$
    is the category $\mathcal{F}_{P}(G)$
    whose objects are the subgroups of $P$
    and morphisms are given by
    $\mathrm{Hom}_{\mathcal{F}_{P}(G)}(Q,R)=\mathrm{Hom}_{G}(Q,R)$.
    For subgroups $Q$ and $R$ of $P$,
    we say that $Q$ and $R$ are \textit{$\mathcal{F}_{P}(G)$-conjugate}
    if $Q$ and $R$ are isomorphic in $\mathcal{F}_{P}(G)$.
    Let $Q$ be a subgroup of $P$.
    We say that $Q$ is \textit{fully automized} in $\mathcal{F}_{P}(G)$
    if $\mathrm{Aut}_{P}(Q)$ is a Sylow $p$-subgroup of
    $\mathrm{Aut}_{\mathcal{F}_{P}(G)}(Q)$.
    We say that $Q$ is \textit{receptive} in $\mathcal{F}_{P}(G)$
    if for any subgroup $R$ of $P$ and
    any $\varphi \in \mathrm{Iso}_{\mathcal{F}_{P}(G)}(R,Q)$,
    there is an element $\bar{\varphi}
    \in \mathrm{Hom}_{\mathcal{F}_{P}(G)}(N_{\varphi},P)$
    such that $\bar{\varphi}|_{Q}=\varphi$,
    where $N_{\varphi}=\{g\in N_{P}(R)\mid {c_{g}}^{{\varphi}^{-1}}
    \in \textrm{Aut}_{P}(Q)\}$.
    We say that $Q$ is \textit{fully normalized} in $\mathcal{F}_{P}(G)$
    if $|N_{P}(Q)|\geq |N_{P}(R)|$ for any subgroup $R$ of $P$ that is
    $\mathcal{F}_{P}(G)$-conjugate to $Q$.
    The fusion system $\mathcal{F}_{P}(G)$ is \textit{saturated}
    if any subgroup of $P$ is $\mathcal{F}_{P}(G)$-conjugate
    to a subgroup that is fully automized and receptive.
    In this paper, we use the following facts:
    \begin{remark}
        \label{important-properties-fusion-system}
        \begin{enumerate}[label={$(\mathrm{\roman*})$}]
            \item By the definition, 
            we can take fully normalized subgroups
            as representatives of $\mathcal{F}_{P}(G)$-conjugacy classes
            of subgroups of $P$.
            
            \item If $\mathcal{F}_{P}(G)$ is saturated,
            then any fully normalized subgroup is
            fully automized and receptive in $\mathcal{F}_{P}(G)$
            $($see \textnormal{\cite[\textit{Theorem }5.2]{Roberts-Shpectorov-2009}}$)$.

            \item If $P$ is a Sylow $p$-subgroup of $G$,
            then $\mathcal{F}_{P}(G)$ is saturated
            $($see \textnormal{\cite[\textit{Proposition }1.3]{Broto-Levi-Oliver-2003}}$)$.
        \end{enumerate}        
    \end{remark}

    \section{Relative Stable Equivalences of Morita Type}
    \label{sec-relative-stm}

    The notion of projectivity relative to a $kG$-module $W$
    was first introduced by Okuyama \cite{Okuyama-preprint}
    (see also \cite[Section 8]{Carlson-Book-1996}).
    In \cite{Carlson-Peng-Wheeler-1998},
    a $W$-stable category $\underline{\textrm{mod}}^{W}(kG)$,
    which is an analogue of the stable category
    $\underline{\textrm{mod}}(kG)$,
    was defined, and it was shown that $\underline{\textrm{mod}}^{W}(kG)$
    is a triangulated category.
    In \cite{Wang-Zhang-2018},
    it was shown that
    for a block $B$ of $kG$,
    the full subcategory $\underline{\textrm{mod}}^{W}(B)$
    of $\underline{\textrm{mod}}^{W}(kG)$
    whose objects are all finitely generated $B$-modules
    is a triangulated subcategory.
    Wang and Zhang \cite{Wang-Zhang-2018} also introduced
    the notion of a relative $(W,W')$-stable equivalence of Morita type
    between blocks $B$ and $B'$ of finite groups $G$ and $G'$, respectively,
    where $W$ is a $kG$-module and $W'$ is a $kG'$-module.
    In this paper, we use the subgroup versions of these notions
    (see below \cref{def-Wang-Zhang}).
    The main purpose of this section is to prove
    \cref{relative-stable-induces-equiv-stable-categories},
    which shows that under suitable hypotheses,
    a relative $Q$-stable equivalence of Morita type
    between blocks $B$ and $B'$ with a common defect group $P$
    induces an equivalence between $\underline{\textrm{mod}}^{Q}(B)$
    and $\underline{\textrm{mod}}^{Q}(B')$
    as triangulated categories,
    where $Q$ is a subgroup of $P$.

    Let $W$ be a $kG$-module.
    We say that a $kG$-module $U$ is relatively $W$-projective
    if $U$ is a direct summand of $W\otimes V$ for some $kG$-module $V$,
    where $W\otimes V$ is considered as a $kG$-module via the diagonal action.
    We define the $W$-stable category 
    \underline{mod}$^{W}$($kG$) of mod($kG$)
    whose objects are the same as those of $\textrm{mod}(kG)$,
    and whose morphisms are given by
    \[
        \underline{\mathrm{Hom}}^{W}_{kG}(U,V)
        =
        \mathrm{Hom}_{kG}(U,V)/\mathrm{Hom}^{W}_{kG}(U,V),
    \]
    where $\mathrm{Hom}^{W}_{kG}(U,V)$ is the subspace of $\mathrm{Hom}_{kG}(U,V)$
    consisting of all homomorphisms each of which factors through
    a $W$-projective $kG$-module.
    For a block $B$ of $kG$,
    we write \underline{mod}$^{W}(B)$ for the full subcategory
    of \underline{mod}$^{W}(kG)$ whose objects
    are all finitely generated $B$-modules.
    We write $\underline{f}$ for the image of
    a homomorphism $f:U\rightarrow V$
    in $\underline{\mathrm{Hom}}^{W}_{kG}(U,V)$ and
    $\underline{\mathrm{Hom}}^{W}_{B}(U,V)$.
    
    We recall from \cite{Carlson-Peng-Wheeler-1998}
    that $W$-stable categories are triangulated.
    We say that a short exact sequence of $kG$-modules
    \[
        E:
        \begin{tikzcd}
            0\ar[r]
            & U_{1}\ar[r,"f"]
            & U_{2}\ar[r,"g"]
            & U_{3}\ar[r]
            & 0
        \end{tikzcd}
    \]
    is $W$-split if $E\otimes W$ splits.
    Then $f$ is called a $W$-split monomorphism,
    and $g$ is called a $W$-split epimorphism.
    We write $\alpha_{W}:W^{*}\otimes W\rightarrow k$ for the evaluation map,
    that is, the homomorphism defined by $\alpha_{W}(f\otimes w)=f(w)$,
    where $W^{*}$ is considered as a right $kG$-module,
    and $W^{*}\otimes W$ is considered as a $kG$-module via the diagonal action.
    Then $\alpha_{W}$ is a $W$-split epimorphism
    (see \cite[Lemma 2.2]{Carlson-Peng-Wheeler-1998}),
    and hence its dual $\alpha^{*}_{W}:k\rightarrow W^{*}\otimes W$
    is a $W$-split monomorphism.
    For a $kG$-module $U$,
    we write $I_{W}(U)=U\otimes W^{*}\otimes W$,
    and write $\Omega_{W}^{-1}(U)$ for the cokernel of
    $\mathrm{id}_{U}\otimes\alpha^{*}_{W}:U\rightarrow I_{W}(U)$.
    For a $kG$-homomorphism $f_{1}:U_{1}\rightarrow U_{2}$,
    we have a commutative diagram of $W$-split short exact sequences
    \[
        \begin{tikzcd}
            0 \ar[r] & U_{1}\ar[r,"\mathrm{id}_{U_{1}}\otimes \alpha^{*}_{W}"]\ar[d,"f_{1}"'] & I_{W}(U_{1})\ar[r]\ar[d] & \Omega_{W}^{-1}(U_{1})\ar[r]\ar[d,equal] & 0\\
            0 \ar[r] & U_{2}\ar[r,"f_{2}"] & U_{3}\ar[r,"f_{3}"] & \Omega_{W}^{-1}(U_{1})\ar[r] & 0.
        \end{tikzcd}
    \]
    Then we get a triangle 
    \begin{tikzcd}
        U_{1}\ar[r,"\underline{f_{1}}"]
        & U_{2}\ar[r,"\underline{f_{2}}"]
        & U_{3}\ar[r,"\underline{f_{3}}"]
        & \Omega_{W}^{-1}(U_{1})
    \end{tikzcd}
    in $\mathrm{\underline{mod}}^{W}(kG)$,
    and it is called a standard triangle.
    We call a triangle in $\mathrm{\underline{mod}}^{W}(kG)$
    isomorphic to a standard triangle a distinguished triangle.
    Let $\mathscr{T}$ be the collection of distinguished triangles 
    in $\mathrm{\underline{mod}}^{W}(kG)$.
    Then it follows that \underline{mod}$^{W}$($kG$)
    with $\Omega_{W}^{-1}$ and $\mathscr{T}$
    is a triangulated category
    (see \cite[Theorem 6.2]{Carlson-Peng-Wheeler-1998}).
    For a block $B$ of $kG$,
    it holds that \underline{mod}$^{W}(B)$ is
    a triangulated subcategory of \underline{mod}$^{W}$($kG$)
    (see \cite[Proposition 3.1]{Wang-Zhang-2018}).
    
    Wang and Zhang \cite{Wang-Zhang-2018} introduced the notion of
    relative stable equivalences of Morita type:
    \begin{definition}\textnormal{(see \cite[Definition 5.1]{Wang-Zhang-2018})}
        \label{def-Wang-Zhang}
        Let $G$ and $G'$ be finite groups
        and $B$ and $B'$ blocks of $G$ and $G'$, respectively.
        Let $W$ be a $kG$-module and $W'$ a $kG'$-module.
        For a $B$-$B'$-bimodule $M$, and a $B'$-$B$-bimodule $N$,
        we say that the pair $(M,N)$ induces
        a relative $(W,W')$-stable equivalence of Morita type between $B$ and $B'$
        if $M$ and $N$ are finitely generated projective as left modules and right modules
        with the property that there are isomorphisms of bimodules 
        \[
            M\otimes_{B'} N\cong B\oplus X
            \text{\quad  and \quad}
            N\otimes_{B} M\cong B'\oplus Y,
        \]
        where $X$ is $W^{*}\otimes W$-projective 
        as a $k[G\times G]$-module
        and $Y$ is ${W'}^{*}\otimes W'$-projective
        as a $k[G'\times G']$-module.
    \end{definition}

    In this paper,
    we mainly consider subgroup versions of the notions above.
    Let $H$ be a subgroup of $G$.
    It follows from the Frobenius reciprocity
    that a $kG$-module $U$ is $H$-projective
    if and only if $U$ is $k_{H}{\uparrow}^{G}$-projective.
    Therefore projectivity relative to modules
    is a generalization of projectivity relative to subgroups.
    We write $\underline{\textrm{mod}}^{H}(kG)
    =\underline{\textrm{mod}}^{k_{H}{\uparrow}^{G}}(kG)$,
    and,
    for a block $B$ of $kG$, we write $\underline{\textrm{mod}}^{H}(B)
    =\underline{\textrm{mod}}^{k_{H}{\uparrow}^{G}}(B)$.
    We say that a short exact sequence of $kG$-modules is $H$-split
    if its restriction to $H$ splits.
    We see that a short exact sequence of $kG$-modules is
    $k_{H}\uparrow^{G}$-split if and only if it is $H$-split.
    In \cref{def-Wang-Zhang},
    suppose further that $B$ and $B'$ have a common defect group $P$.
    Then for a subgroup $Q$ of $P$,
    we say that $(M,N)$ induces a relative $Q$-stable equivalence of Morita type
    between $B$ and $B'$
    if $(M,N)$ induces a relative $(W,W')$-stable equivalence of Morita type
    with $W=k_{Q}\uparrow^{G}$ and $W'=k_{Q}\uparrow^{G'}$.
    With this definition, $X$ and $Y$ are $Q\times Q$-projective
    since it follows that
    \[
        {(k_{Q}\uparrow^{G})}^{*}\otimes k_{Q}\uparrow^{G}
        \cong kG\otimes_{kQ}k_{Q}\otimes k_{Q}\otimes_{kQ}kG
        \cong k_{Q\times Q}\uparrow^{G\times G}.
    \]

    Note that if $(M,N)$ induces
    a relative $(W,W')$-stable equivalence of Morita type
    between $B$ and $B'$,
    then $-\otimes_{B}M$ and $-\otimes_{B'}N$ do not, in general,
    induce an equivalence between $\underline{\textrm{mod}}^{W}(B)$
    and $\underline{\textrm{mod}}^{W'}(B')$.
    Indeed, suppose that $(M,N)$ induces a stable equivalence of Morita type
    between $B$ and $B'$.
    Then $X$ and $Y$ are $1$-projective,
    and hence $X$ is $Q\times Q$-projective for some nontrivial $p$-subgroup $Q$ of $G$.
    This means that $(M,N)$ induces
    a $(k_{Q}\uparrow^{G},k_{1}\uparrow^{G'})$-stable equivalence of Morita type.
    However, $-\otimes_{B}M$ sends indecomposable $B$-modules with vertex $Q$,
    which are zero objects in $\underline{\textrm{mod}}^{k_{Q}\uparrow^{G}}(B)$,
    to nonprojective $B'$-module.
    Hence $-\otimes_{B}M$ dose not induce an equivalence between
    $\underline{\textrm{mod}}^{k_{Q}\uparrow^{G}}(B)$
    and $\underline{\textrm{mod}}^{k_{1}\uparrow^{G'}}(B')
    =\underline{\textrm{mod}}(B')$.

    However, we can show that under suitable hypotheses,
    a relative $Q$-stable equivalence of Morita type
    between $B$ and $B'$ with a common defect group $P$
    induces an equivalence
    between $\underline{\textrm{mod}}^{Q}(B)$ and $\underline{\textrm{mod}}^{Q}(B')$
    as triangulated categories,
    where $Q$ is a subgroup of $P$.
    In order to show this,
    the following lemmas are needed.

    \begin{lemma}\textnormal{(see \cite[Lemma 9.4]{Okuyama-preprint})}
        \label{equiv-cond-relative-split}
        Let $W$ be a $kG$-module and
        \[
            E:\begin{tikzcd}
                0\arrow[r]
                & U_{1}\arrow[r,"f_{1}"]
                & U_{2}\arrow[r,"f_{2}"]
                & U_{3}\arrow[r]
                & 0
            \end{tikzcd}
        \]
        a short exact sequence of $kG$-modules.
        Then the following are equivalent:
        \begin{enumerate}[label={$(\mathrm{\roman*})$}]
            \item $E$ is $W$-split.
            \item For any $W$-projective $kG$-module $X$,
            the functor $\mathrm{Hom}_{kG}(X,-)$ is exact.

            \item For any $W$-projective $kG$-module $Y$,
            the functor $\mathrm{Hom}_{kG}(-,Y)$ is exact.
        \end{enumerate}
    \end{lemma}

    \begin{lemma}
        \label{basic-properties-Delta-projective-p-perm-tensor-functor}
        Let $G$ and $G'$ be finite groups with a common $p$-subgroup $P$
        such that $\mathcal{F}_{P}(G)=\mathcal{F}_{P}(G')$.
        Let $M$ be a $\Delta P$-projective $p$-permutation
        $k[G\times G']$-module,
        and $Q$ a subgroup of $P$.
        Then the following hold:
        \begin{enumerate}[label={$(\mathrm{\roman*})$}]
            \item If $U$ is a $Q$-projective $kG$-module,
            then $U\otimes_{kG}M$ is $Q$-projective.
            \item If a short exact sequence of $kG$-modules
            \[
                \begin{tikzcd}
                    E:0\arrow[r]
                    & U_{1}\arrow[r,"f_{1}"]
                    & U_{2}\arrow[r,"f_{2}"]
                    & U_{3}\arrow[r]
                    & 0
                \end{tikzcd}
            \]
            is $Q$-split, then $E\otimes_{kG}M$ is $Q$-split.
            \end{enumerate}
            Moreover,
            if $B$ and $B'$ are blocks of $G$ and $G'$,
            respectively, with a common defect group $P$,
            and $M$ is a $B$-$B'$-bimodule
            that is a $\Delta P$-projective $p$-permutation
            $k[G\times G']$-module,
            then the following holds.

            \begin{enumerate}[label={$(\mathrm{\roman*})$}]
            \setcounter{enumi}{2}
            \item The functor ${-}\otimes_{B}M$ induces
            a functor of triangulated categories
            $\underline{\textnormal{mod}}^{Q}(B)
            \rightarrow\underline{\textnormal{mod}}^{Q}(B')$.
        \end{enumerate}

    \end{lemma}
    \begin{proof}
        (i):
        Let $U$ be a $Q$-projective $kG$-module.
        Then $U$ is a direct summand of $V{\uparrow}_{Q}^{G}$
        for some $kG$-module $V$.
        Since $M$ is a $\Delta P$-projective $p$-permutation module,
        $M$ is a direct summand of
        $(\bigoplus_{i}k_{Q_{i}}{\uparrow}^{P}){\uparrow}_{\Delta P}^{G\times G'}$,
        where $\bigoplus_{i}k_{Q_{i}}{\uparrow}^{P}$ is considered
        as a $k\Delta P$-module via the isomorphism $P\cong \Delta P$.
        We see that
        \[
            M\mid
            (\bigoplus_{i}k_{Q_{i}}{\uparrow}^{P}){\uparrow}_{\Delta P}^{G\times G'}
            \cong (kG\otimes (\bigoplus_{i}k_{Q_{i}}{\uparrow}^{P}))
            \otimes_{kP}kG'
            \cong \bigoplus_{i}kG\otimes_{kQ_{i}}kG'
        \]
        Hence we have that
        \begin{align*}
            U\otimes_{kG}M
            \mid V{\uparrow}_{Q}^{G}\otimes_{kG}M
            \mid V{\uparrow}_{Q}^{G}\otimes_{kG}(\bigoplus_{i}kG\otimes_{kQ_{i}}kG')
            &\cong \bigoplus_{i}V{\uparrow}_{Q}^{G}{\downarrow}_{Q_{i}}^{}{\uparrow}_{}^{G'}\\
            &\cong \bigoplus_{i}\bigoplus_{t\in [Q\backslash G/Q_{i}]}
            V^{t}{\uparrow}_{Q^{t}\cap Q_{i}}^{G'}.
        \end{align*}
        Hence any indecomposable summand of $U\otimes_{kG}M$ is
        $Q^{t}\cap Q_{i}$-projective for some element $t\in G$.
        Then $t$ induces a conjugation map
        $Q\cap Q_{i}^{t^{-1}}\rightarrow Q^{t}\cap Q_{i}$ in $\mathcal{F}_{P}(G)$.
        By the assumption that $\mathcal{F}_{P}(G)=\mathcal{F}_{P}(G')$,
        there is an element $s\in G'$ such that $Q^{s}\cap Q_{i}^{t^{-1}s}=Q^{t}\cap Q_{i}$.
        This implies that $U\otimes_{kG}M$ is $Q$-projective.

        (ii):
        Let $Y$ be any $Q$-projective $kG'$-module.
        Since $M$ is projective as a left $kG$-module,
        the functor $-\otimes_{kG}M$ is right adjoint to
        $-\otimes_{kG'}M^{*}$.
        Hence we have the following commutative diagram
        \[
            \begin{tikzcd}
                0\arrow[r]
                & \mathrm{Hom}_{kG}(U_{3},Y\otimes_{kG'}M^{*})
                \arrow[r,"{f_{2}}^{*}"]\arrow[d,"\rotatebox{90}{$\sim$}"]
                & \mathrm{Hom}_{kG}(U_{2},Y\otimes_{kG'}M^{*})
                \arrow[r,"{f_{1}}^{*}"]\arrow[d,"\rotatebox{90}{$\sim$}"]
                & \mathrm{Hom}_{kG}(U_{1},Y\otimes_{kG'}M^{*})
                \arrow[r]\arrow[d,"\rotatebox{90}{$\sim$}"]
                & 0
                \\
                0\arrow[r]
                & \mathrm{Hom}_{kG}(U_{3}\otimes_{kG}M,Y)
                \arrow[r,"{(f_{2}\otimes \mathrm{id}_{M})}^{*}"]
                & \mathrm{Hom}_{kG}(U_{2}\otimes_{kG}M,Y)
                \arrow[r,"{(f_{1}\otimes \mathrm{id}_{M})}^{*}"]
                & \mathrm{Hom}_{kG}(U_{1}\otimes_{kG}M,Y)
                \arrow[r]
                & 0.
            \end{tikzcd}
        \]
        By (i), $Y\otimes_{kG'}M^{*}$ is $Q$-projective,
        and hence by \cref{equiv-cond-relative-split},
        the first row in the diagram above is exact.
        This implies that $E\otimes M$ is $Q$-split again by \cref{equiv-cond-relative-split}.
        
        (iii):
        It follows from (i) that ${-}\otimes_{B}M$ induces
        a functor $\underline{\textrm{mod}}^{Q}(B)
        \rightarrow\underline{\textrm{mod}}^{Q}(B')$.
        By \cite[Proposition 6.3]{Carlson-Peng-Wheeler-1998},
        every distinguished triangle in $\underline{\textrm{mod}}^{W}(B)$
        is isomorphic to that arising from a $W$-split short exact sequence.
        Therefore we can show that the functor $\underline{\textrm{mod}}^{Q}(B)
        \rightarrow\underline{\textrm{mod}}^{Q}(B')$ induced by $-\otimes_{B}M$
        is a functor of triangulated categories
        in the same way as for the stable categories.
    \end{proof}

    Now we show the main result of this section.
    
    \begin{proposition}
        \label{relative-stable-induces-equiv-stable-categories}
        Let $B$ and $B'$ be blocks of finite groups $G$ and $G'$,
        respectively, with a common defect group $P$
        such that $\mathcal{F}_{P}(G)=\mathcal{F}_{P}(G')$.
        Let $M$ be a $B$-$B'$-bimodule
        that is a $\Delta P$-projective $p$-permutation module
        as a $k[G\times G']$-module,
        and $N$ a $B'$-$B$-bimodule
        that is a $\Delta P$-projective $p$-permutation module
        as a $k[G'\times G]$-module.
        Let $Q$ be a subgroup of $P$.
        If $(M,N)$ induces a relative $Q$-stable equivalence of Morita type,
        then ${-}\otimes_{B}M$ and ${-}\otimes_{B'}N$ are equivalences between
        $\underline{\textnormal{mod}}^{Q}(B)$ and $\underline{\textnormal{mod}}^{Q}(B')$
        as triangulated categories.
    \end{proposition}
    \begin{proof}
        It follows from \cref{basic-properties-Delta-projective-p-perm-tensor-functor} (iii)
        that $-\otimes_{B}M$ and $-\otimes_{B'}N$ induce functors
        of triangulated categories between $\underline{\textrm{mod}}^{Q}(B)$
        and $\underline{\textrm{mod}}^{Q}(B')$.
        Therefore it suffices to show that the functors are equivalences.
        For a $B$-$B'$-bimodule $X$
        that is $Q\times Q$-projective as a $k[G\times G']$-module,
        and a $kG$-module $U$,
        we see that $U\otimes_{B}X$ is $Q$-projective.
        Hence the result follows from the same argument as for the stable categories.
    \end{proof}

    We end this section with a remark on the definition of 
    the relative stable category \underline{mod}$^{W}(B)$.
    For \underline{mod}$^{W}(B)$,
    the $kG$-module $W$ does not necessarily lie in $B$
    since a $kG$-module lying in $B$ may be projective
    relative to modules lying in blocks other than $B$.
    Indeed, suppose that $B$ is a nonprincipal block of $G$,
    and $S$ is a simple $B$-module.
    Then $P(S)$ is a direct summand of $P(k_{G})\otimes S$,
    where $P(S)$ and $P(k_{G})$ are projective covers of $S$ and $k_{G}$,
    respectively.
    This means that $P(S)$, which lies in $B$,
    is projective relative to $P(k_{G})$,
    which lies in $B_{0}(G)$.
    In fact, this observation holds for any $kG$-module,
    not just for projective modules:
    \begin{remark}
        For a subgroup $H$, a $kG$-module $U$ is $k_{H}{\uparrow}^{G}$-projective
        $($or equivalently $H$-projective$)$
        if and only if $U$ is projective relative to $S(G,H)$,
        which lies in $B_{0}(G)$.
        Indeed, if $U$ is $S(G,H)$-projective,
        then it follows from the Frobenius reciprocity that
        $U$ is $H$-projective.
        Conversely, suppose that $U$ is $H$-projective.
        There is an $H$-split epimorphism $S(G,H)\rightarrow k_{G}$
        since $S(G,H)$ is a relative $H$-projective cover of $k_{G}$
        $($see \textnormal{\cite[\textit{Proposition }3.1]{Thevenaz-1985}}$)$.
        Then $S(G,H)\otimes U\rightarrow U$ is still an $H$-split epimorphism.
        Since $U$ is $H$-projective, the epimorphism splits,
        and hence $U$ is $S(G,H)$-projective.
    \end{remark}

    \section{Proofs of \cref{generalization-Broue-gluing} and \cref{generalization-Linckelmann}}

    In this section, we prove
    \cref{generalization-Broue-gluing} and \cref{generalization-Linckelmann}.

    We need the following two lemmas for the proof of \cref{generalization-Broue-gluing}.

    \begin{lemma}\textnormal{(see \cite[Lemma 3.3]{Koshitani-Lassueur-2020})}
        \label{equivalent-condition-having-Scott-as-direct-summand}
        Let $G$ and $G'$ be finite groups
        with a common Sylow $p$-subgroup $P$
        such that $\mathcal{F}_P(G)=\mathcal{F}_P(G')$,
        $M$ a $\Delta P$-projective $p$-permutation $k[G\times G']$-module,
        and $Q$ a subgroup of $P$.
        Then the following are equivalent:
        \begin{enumerate}[label={$(\mathrm{\roman*})$}]
            \item The Scott module $S(G',Q)$ is a direct summand of $k_{G}\otimes_{kG}M$.
            \item The Scott module $S(G\times G',\Delta Q)$ is a direct summand of $M$.
        \end{enumerate}
    \end{lemma}

    Although we may see the following lemma
    from \cite[Proposition 4.6]{Linckelmann-2015},
    we show it for the convenience of the reader.

    \begin{lemma}\label{block-direct-summand}
        Let $G$ and $G'$ be finite groups
        with a common Sylow $p$-subgroup $P$
        such that $\mathcal{F}_{P}(G)=\mathcal{F}_{P}(G')$,
        and let $M=S(G\times G',\Delta P)$.
        Then there is an isomorphism of $B_{0}(G)$-$B_{0}(G)$-bimodules
        \[
            M\otimes_{B_{0}(G')}M^{*}\cong B_{0}(G)\oplus X,
        \]
        where $X$ is a $\Delta P$-projective $p$-permutation module
        as a $k[G\times G]$-module.
    \end{lemma}
    \begin{proof}
        Let $B=B_{0}(G)$ and $B'=B_{0}(G')$.

        Let $X'$ be an indecomposable summand of $M\otimes_{B'}M^{*}$.
        Then by \cite[Theorem 5.1.16]{Linckelmann-book-2018-vol1},
        $X'$ has a vertex $R$ that is a subgroup of
        $\Delta_{t}(P\cap P^{t^{-1}})
        :=\{(x,x^{t}) \mid x\in P\cap P^{t^{-1}}\}$ for some $t\in G'$,
        and a source that is isomorphic to
        $(k_{P}\otimes k_{P^{t^{-1}}}){\downarrow}_{R}^{\Delta_{t}(P\cap P^{t^{-1}})}
        \cong k_{R}$,
        where $k_{P}\otimes k_{P^{t^{-1}}}$ is considered
        as $k\Delta_{t}(P\cap P^{t^{-1}})$-module via the isomorphism
        $P\cap P^{t^{-1}}\cong \Delta_{t}(P\cap P^{t^{-1}})$.
        Hence $M\otimes_{B'}M^{*}$ is a $p$-permutation module.
        We see that $t$ induces a conjugation map
        $P\cap P^{t^{-1}}\rightarrow P^{t}\cap P$ in $\mathcal{F}_{P}(G')$.
        By the assumption that $\mathcal{F}_P(G)=\mathcal{F}_P(G')$,
        there is an element $s\in G$
        such that $x^{s}=x^{t}$ for any $x\in P\cap P^{t^{-1}}$.
        Hence we have that
        \[
            R^{^{(s,1)}}
            \leq\{(x^{s},x^{t})\mid x\in P\cap P^{t^{-1}}\}
            =\Delta(P^{t}\cap P),
            \]
        which implies that $R\leq_{G\times G}\Delta P$.
        Thus $M\otimes_{B'}M^{*}$ is a $\Delta P$-projective $p$-permutation module.

        Since $P$ is a Sylow $p$-subgroup of $G'$,
        it follows that $S(G',P)\cong S(G',G')=k_{G'}$.
        Hence we see, using \cref{equivalent-condition-having-Scott-as-direct-summand} twice,
        that $(k_{G}\otimes_{B} M)\otimes_{B'}M^{*}\cong k_{G}\oplus Y$
        for some $B$-module $Y$.
        Hence again by \cref{equivalent-condition-having-Scott-as-direct-summand}, 
        we see that $S(G\times G,\Delta P)\cong B$ is a direct summand of
        $M\otimes_{B'}M^{*}$.
    \end{proof}

    We now prove \cref{generalization-Broue-gluing}.
    
    \begin{proof}[Proof of \cref{generalization-Broue-gluing}]
        Note that in the proof, we use the isomorphism
        \[
            (M\otimes_{B'}M^*)(\Delta Q)
            \cong
            M(\Delta Q)\otimes_{B_{0}(C_{G'}(Q))}M(\Delta Q)^*
        \]
        for any subgroup $Q$ of $P$(see \cite[proof of Theorem 4.1]{rickard-1996}).

        Let $B=B_0(G)$ and $B'=B_0(G')$.

        $(\mathrm{ii})\Rightarrow(\mathrm{i})$:
        We can write $M\otimes_{B'}M^{*}\cong B\oplus X$,
        where $X$ is a $B$-$B$-bimodule that is $Z\times Z$-projective
        as a $k[G\times G]$-module.
        Hence we have that for any subgroup $Q$ of $P$,
        \[
            M(\Delta Q)\otimes_{B_{0}(C_{G'}(Q))}M(\Delta Q)^*
            \cong (M\otimes_{B'}M^*)(\Delta Q)
            \cong (B\oplus X)(\Delta Q)
            \cong B_{0}(C_{G}(Q))\oplus X(\Delta Q).
        \]
        Since $X$ is $Z\times Z$-projective,
        it follows from \cref{Brauer-construction-vertex} that
        if $Q$ is properly containing $Z$,
        then $X(\Delta Q)=0$.
        Similarly, we see that 
        $M(\Delta Q)^{*}\otimes_{B_{0}(C_{G}(Q))}M(\Delta Q)\cong B_{0}(C_{G'}(Q))$
        for any subgroup $Q$ of $P$ properly containing $Z$.

        $(\mathrm{i})\Rightarrow(\mathrm{ii})$:
        By \cref{block-direct-summand},
        we can write $M\otimes_{B'}M^{*}\cong B\oplus X$,
        where $X$ is a $\Delta P$-projective $p$-permutation $k[G\times G']$-module.
        To show that $X$ is $Z\times Z$-projective,
        we show that $X$ is $\Delta Z$-projective.
        Let $X'$ be any indecomposable summand of $X$.
        Since $\Delta Z$ is a $p$-subgroup of the center $Z(G\times G')$
        contained in $\Delta P$,
        it follows from \cref{Brauer-construction-trivial} that
        \[
            (M\otimes_{B'}M^*)(\Delta Z)
            \cong M(\Delta Z)\otimes_{B_{0}(C_{G'}(Z))}M(\Delta Z)^*
            \cong M\otimes_{B'}M^{*}.
        \]
        On the other hand, we have that 
        \[
            (M\otimes_{B'}M^*)(\Delta Z)
            = (B\oplus X)(\Delta Z)
            \cong B(\Delta Z)\oplus X(\Delta Z)
            \cong B\oplus X(\Delta Z),
        \]
        where the last isomorphism holds
        as $B(\Delta Z)\cong B$ by \cref{Brauer-construction-trivial}.
        Hence it follows that $X(\Delta Z)\cong X$.
        This implies that $X'(\Delta Z)\neq 0$,
        and hence $X'$ has a vertex $R$ containing $\Delta Z$
        by \cref{Brauer-construction-vertex} (i).
        Now, we have that $\Delta Z\leq R\leq_{G\times G}\Delta P$.
        This means that if we can show that
        $X'(\Delta Q)=0$ for any subgroup $Q$ of $P$ properly containing $Z$,
        then it follows from \cref{Brauer-construction-vertex} (ii) that $R=\Delta Z$.
        Therefore, let $Q$ be any subgroup of $P$ properly containing $Z$.
        Since $(M(\Delta Q),M(\Delta Q)^*)$ induces a Morita equivalence
        between $B_0(C_G(Q))$ and $B_0(C_{G'}(Q))$, we have that
        \[
            (M\otimes_{B'}M^*)(\Delta Q)
            \cong M(\Delta Q)\otimes_{B_{0}(C_{G'}(Q))}M(\Delta Q)^*
            \cong B_{0}(C_{G}(Q)).    
        \]
        On the other hand, we have that
        \[
            (M\otimes_{B'}M^*)(\Delta Q)\cong B_{0}(C_{G}(Q))\oplus X(\Delta Q).
        \]
        Hence we see that $X(\Delta Q)=0$,
        and in particular, $X'(\Delta Q)=0$.
        Finally, we have that $X$ is $\Delta Z$-projective.
        Similarly, if we write $M^{*}\otimes_{B}M\cong B'\oplus Y$,
        then $Y$ is $\Delta Z$-projective.
    \end{proof}

    Let $B$ be a block of $G$ with defect group $P$,
    and $S$ a simple $kG$-module lying in $B$.
    Then $P\cap Z(G)$ is contained in a vertex of $S$
    since $P\cap Z(G)$ acts trivially on $S$.
    The following lemma gives a condition for the vertex
    to be equal to $P\cap Z(G)$.

    \begin{lemma}\label{generalization-equiv-cond-block-simple}
        Let $B$ be a block of $kG$ with defect group $P$.
        Then the following are equivalent:
        \begin{enumerate}[label={$(\mathrm{\roman*})$}]
            \item The block $B$ has a simple module with vertex $P\cap Z(G)$.
            \item $P\leq Z(G)$.
        \end{enumerate}
        Moreover, if these conditions hold, then $B$ has a unique simple module.
    \end{lemma}
    \begin{proof}
        $(\mathrm{ii})\Rightarrow(\mathrm{i})$:
        Since $P$ is a normal $p$-subgroup of $G$,
        it follows that $P$ acts trivially on any simple $B$-module $S$.
        Hence by \cite[Chapter 4, Theorem 7.8 (i)]{Nagao-Tsushima-book-1989},
        that $S$ has a vertex containing $P$.
        This implies (i).

        $(\mathrm{i})\Rightarrow(\mathrm{ii})$:
        Let $Z=P\cap Z(G)$ and $S$ a simple $B$-module with vertex $Z$.
        Hence $S$ is projective as a $k[G/Z]$-module.
        Let $\bar{B}$ be the block of defect zero of $k[G/Z]$
        in which $S$ lies.
        Then $\bar{B}$ has a unique irreducible character,
        and it lies in $B$ when viewed as a character of $G$.
        Hence we see from \cite[Chapter 5, Lemma 8.6 (ii)]{Nagao-Tsushima-book-1989}
        that $B$ dominates $\bar{B}$.
        Also, by \cite[Chapter 5, Theorem 8.11]{Nagao-Tsushima-book-1989},
        $\bar{B}$ is a unique block of $k[G/Z]$ dominated by $B$.
        It follows from \cite[Chapter 5, Theorem 8.10]{Nagao-Tsushima-book-1989} that $P/Z=1$.
        Hence (ii) follows.
        
        Suppose that the conditions hold.
        In the argument above,
        since $\bar{B}$ has a unique simple module,
        so does $B$ by \cite[Chapter 5, Theorem 8.11]{Nagao-Tsushima-book-1989}.
    \end{proof}

    The following proposition is a key result for the proof of
    \cref{generalization-Linckelmann}.
    This is a generalization of \cite[Proposition 2.3]{Linckelmann-1996}
    to projectivity relative to a central $p$-subgroup under certain conditions.
    \begin{proposition}\label{generalization-Linckelmann-proposition}
        Let $G$ and $G'$ be finite groups
        with a common $p$-subgroup $P$,
        and $M$ a trivial source $k[G\times G']$-module with vertex $\Delta P$.
        Assume that $Z$ is a proper subgroup of $P$ contained in $Z(G)$.
        Then $\mathrm{soc}(kG)\otimes_{kG}M$ is a nonzero $kG'$-module
        having no nonzero $Z$-projective summand.
        In particular,
        for any simple $kG$-module $S$,
        the $kG'$-module $S\otimes_{kG}M$ has no nonzero $Z$-projective summand.
    \end{proposition}
    \begin{proof}
        Since $Z$ is a normal $p$-subgroup of $G$,
        it follows that $Z$ acts trivially on any simple $kG$-module.
        Also, it follows that $mz=zm$ for any element $m\in M$ and $z\in Z$
        since $M$ is a direct summand of $kG\otimes_{kP}kG'$.
        Hence $Z$ acts trivially on soc$(kG)\otimes_{kG}M$.
        This implies that any indecomposable summand of soc$(kG)\otimes_{kG}M$
        has a vertex containing $Z$
        (see \cite[Chapter 4, Theorem 7.8 (i)]{Nagao-Tsushima-book-1989}).
        Therefore if we can show that
        $\mathrm{soc}(kG)\otimes_{kG}M$ has 
        no nonzero projective summand as a $k[G'/Z]$-module,
        then the result follows.
        Let $\pi_{Z}:kG'\rightarrow k[G'/Z]$ be
        the canonical algebra homomorphism.
        Then this is equivalent to saying that
        $(\mathrm{soc}(kG)\otimes_{kG}M)\mathrm{soc}(k[G'/Z])=0$
        by \cite[Proposition 4.11.7]{Linckelmann-book-2018-vol1},
        and we have the isomorphisms
        \begin{align*}
            (\mathrm{soc}(kG)\otimes_{kG}M)\mathrm{soc}(k[G'/Z])
            &\cong(\mathrm{soc}(kG)M)\pi_{Z}^{-1}(\mathrm{soc}(k[G'/Z]))\\
            &\cong M\mathrm{soc}(kG^{\mathrm{op}})\otimes\pi_{Z}^{-1}(\mathrm{soc}(k[G'/Z])),
        \end{align*}
        where $kG^{\mathrm{op}}$ is the opposite algebra of $kG$.
        Hence we show that
        \[
            M\mathrm{soc}(kG^{\mathrm{op}})\otimes\pi_{Z}^{-1}(\mathrm{soc}(k[G'/Z]))=0.  
        \]

        Since $\Delta Z$ acts trivially on $M$,
        it follows that $M$ can be viewed as a $k[G\times G'/\Delta Z]$-module
        with vertex $\Delta P/\Delta Z$.
        We may consider $kG^{\mathrm{op}}\otimes_{kZ}kG'$ as a $k$-algebra
        since $kG^{\mathrm{op}}$ and $kG'$ are algebras over the commutative ring $kZ$.
        Let $\theta:kG^{\mathrm{op}}\otimes_{kZ}kG'\rightarrow k[G\times G'/\Delta Z]$
        be a $k$-algebra homomorphism given by
        $\theta(g^{\mathrm{op}}\otimes g')=(g^{-1},g')\Delta Z$.
        Then $\theta$ is an isomorphism.
        Note here that the opposite of $kG$ in the definition of $\theta$
        is necessary to make $\theta$ an algebra homomorphism.
        We can consider the following commutative diagram of algebras
        \[
            \begin{tikzcd}
                kG^{\mathrm{op}}\otimes kG'\arrow[r,"\sim"]\arrow[d,"\pi"']
                & k[G\times G']\arrow[r]\arrow[d,"\pi_{\Delta Z}"]
                & \mathrm{End}_{k}(M)\\
                kG^{\mathrm{op}}\otimes_{kZ}kG'\arrow[r,"\sim","\theta"']
                & k[G\times G'/\Delta Z]\arrow[ru]
            \end{tikzcd},
        \]
        where $\pi$ is a surjective algebra homomorphism given by
        $\pi(g^{\mathrm{op}}\otimes g')=g^{\mathrm{op}}\otimes g'$,
        the top horizontal map in the square is an isomorphism that
        maps $g^{\mathrm{op}}\otimes g'$ to $g^{-1}\otimes g'$,
        and $\pi_{\Delta Z}$ is the canonical algebra homomorphism.
        Since $P$ is a vertex of $M$, and $\Delta P/\Delta Z$ is nontrivial,
        it follows that $M\mathrm{soc}(k[G\times G'/\Delta Z])=0$.
        Therefore to show that
        $M\mathrm{soc}(kG^{\mathrm{op}})\otimes\nobreak\pi_{Z}^{-1}(\mathrm{soc}(k[G'/Z]))=0$,
        it suffices to show that
        \[
            \pi(\mathrm{soc}(kG^{\mathrm{op}})\otimes\pi_{Z}^{-1}(\mathrm{soc}(k[G'/Z])))
            =\mathrm{soc}(kG^{\mathrm{op}})\otimes_{kZ}\pi_{Z}^{-1}(\mathrm{soc}(k[G'/Z]))
        \]
        is contained in $\theta^{-1}(\mathrm{soc}(k[G\times G'/\Delta Z]))
        =\mathrm{soc}(kG^{\mathrm{op}}\otimes_{kZ}kG')$.

        Since $\Delta Z$ is a normal $p$-subgroup of $G\times G'$,        
        it follows that $\pi_{\Delta Z}(J(k[G\times G']))=\linebreak J(k[G\times G'/\Delta Z])$,
        and hence by the diagram above,
        $\pi(J(kG^{\mathrm{op}}\otimes kG'))=J(kG^{\mathrm{op}}\otimes_{kZ}kG')$.
        We see that
        \begin{align*}
            J(kG^{\mathrm{op}}\otimes_{kZ}kG')
            &=\pi(J(kG^{\mathrm{op}}\otimes kG'))\\
            &=\pi(J(kG^{\mathrm{op}})\otimes kG'+kG^{\mathrm{op}}\otimes J(kG'))\\
            &=J(kG^{\mathrm{op}})\otimes_{kZ}kG'+kG^{\mathrm{op}}\otimes_{kZ}J(kG').
        \end{align*}
        Let $x^{\mathrm{op}}\otimes y\in
        \mathrm{soc}(kG^{\mathrm{op}})\otimes_{kZ}\pi_{Z}^{-1}(\mathrm{soc}(k[G'/Z]))$.
        For $\alpha^{\mathrm{op}}\otimes\beta'\in J(kG^{\mathrm{op}})\otimes_{kZ}kG'$,
        it follows that $(x^{\mathrm{op}}\otimes y)\alpha^{\mathrm{op}}\otimes\beta'=0$
        as $x^{\mathrm{op}}\in \mathrm{soc}(kG^{\mathrm{op}})$
        and $\alpha^{\mathrm{op}}\in J(kG^{\mathrm{op}})$.
        Let $\beta^{\mathrm{op}}\otimes\alpha'\in kG^{\mathrm{op}}\otimes_{kZ}J(kG')$,
        and let us write $y\alpha'=\sum_{t\in[G'/Z]}(\sum_{z\in Z}\lambda_{tz}z)t$.
        Similarly, it follows that $\pi_{Z}(y\alpha')=\pi_{Z}(y)\cdot\pi_{Z}(\alpha')=0$,
        and hence that $\sum_{z\in Z}\lambda_{tz}=0$ for any $t\in[G'/Z]$.
        Hence it follows that
        \begin{align*}
            (x^{\mathrm{op}}\otimes y)\cdot \beta^{\mathrm{op}}\otimes\alpha'
            &=(\beta x)^{\mathrm{op}}\otimes y\alpha'\\
            &=(\beta x)^{\mathrm{op}}\otimes \sum_{t\in[G'/Z]}(\sum_{z\in Z}\lambda_{tz}z)t\\
            &=\sum_{t\in[G'/Z]}(\beta x)^{\mathrm{op}}(\sum_{z\in Z}\lambda_{tz}z)\otimes t\\
            &=\sum_{t\in[G'/Z]}(\beta x)^{\mathrm{op}}(\sum_{z\in Z}\lambda_{tz})\otimes t\\
            &=0,
        \end{align*}
        where the second equality from the last holds
        as $Z$ acts trivially on $\mathrm{soc}(kG^{\mathrm{op}})$.
        Thus\linebreak$\mathrm{soc}(kG^{\mathrm{op}})\otimes_{kZ}\pi_{Z}^{-1}(\mathrm{soc}(k[G'/Z]))$
        is annihilated by $J(kG^{\mathrm{op}}\otimes_{kZ}kG')$,
        and the result follows.
    \end{proof}

    We now prove \cref{generalization-Linckelmann}.
    It can be proved by an argument similar to that in
    \cite[Remark 2.7]{Linckelmann-1996}
    by virtue of \cref{generalization-Linckelmann-proposition}.
    \begin{proof}[Proof of \cref{generalization-Linckelmann}]
        We write
        \[
            M\otimes_{B'}M^{*}\cong B\oplus X
            \text{ and }
            M^{*}\otimes_{B}M\cong B'\oplus Y,
        \]
        where $X$ is $Q\times Q$-projective as a $k[G\times G]$-module,
        and $Y$ is $Q\times Q$-projective as a $k[G'\times G']$-module.
        
        (i) Let $M=M_{1}\oplus M_{2}$,
        where $M_{1}$ and $M_{2}$ are $B$-$B'$-bimodules.
        Then we have that
        \[
            B\oplus X
            =(M_{1}\oplus M_{2})\otimes_{B'}M^{*}
            \cong (M_{1}\otimes_{B'}M^{*})\oplus (M_{2}\otimes_{B'}M^{*}).
        \]
        We may consider that $B$ is a direct summand of $M_{1}\otimes_{B'}M^{*}$.
        Then $M_{2}\otimes_{B'}M^{*}$ is $Q\times Q$-projective.
        Since $M_{2}\otimes_{B'}M^{*}$ is $Q$-projective as a right $kG$-module,
        it follows from
        \cref{basic-properties-Delta-projective-p-perm-tensor-functor} (i),
        $M_{2}\otimes_{B'}M^{*}\otimes_{B}M$ is $Q$-projective
        as a right $kG$-module.
        Hence $M_{2}\otimes_{B'}M^{*}\otimes_{B}M$ is $Q\times Q$-projective.
        We see that
        \[
            M_{2}\otimes_{B'}M^{*}\otimes_{B}M
            = M_{2}\otimes_{B'}(B'\oplus Y)
            \cong M_{2}\oplus (M_{2}\otimes_{B'}Y).
        \]
        Hence $M_{2}$ is $Q\times Q$-projective.

        (ii) Let $S$ be any simple $B$-module.
        First note that $B$ has no $Z$-projective simple module.
        Indeed, $S$ has a vertex containing $Z(G)\cap P$.
        Hence if $Z<Z(G)\cap P$, then clearly $S$ is not $Z$-projective.
        If $Z=Z(G)\cap P$, then $Z(G)\cap P\neq P$ by the assumption on $Z$,
        and hence $S$ is not $Z$-projective by \cref{generalization-equiv-cond-block-simple}.
        
        By the remark above, and the fact that
        $(M,M^{*})$ induces a $Z$-stable equivalence of Morita type,
        we can write $S\otimes_{B}M=V\oplus Y$,
        where $V$ is an indecomposable non $Z$-projective $B'$-module,
        and $Y$ is a $Z$-projective module.
        However, by \cref{generalization-Linckelmann-proposition},
        $S\otimes_{B}M$ has no nonzero $Z$-projective summand,
        and hence $Y=0$.

        (iii) Suppose that $S\otimes_{B}M$ is simple
        for any simple $B$-module $S$.
        It suffices to show that $X=0$
        (see \cite[Theorem 2.1]{rickard-1996},
        and the proof of \cite[Theorem 4.14.10]{Linckelmann-book-2018-vol1}).
        Since $S\otimes_{B}M$ is simple, it follows from (ii) that
        $S\otimes_{B}M\otimes_{B'}M^{*}$ is indecomposable and
        non $Z$-projective.
        On the other hand, we see that
        \[
            S\otimes_{B}M\otimes_{B'}M^{*}
            \cong S\otimes_{B}(B\oplus X)
            \cong S\oplus (S\otimes_{B}X).
        \]
        Hence $S\otimes_{B}X=0$.
        Since $X$ is projective as a $B$-module,
        it follows that $0=S\otimes_{B}X\cong \mathrm{Hom}_{B}(X^{*},S)$.
        This forces $X=0$.
    \end{proof}

    \section{Proof of \cref{generalization-Ishioka-Kunugi}}
    
    In this section, we define relative Brauer indecomposability,
    and prove \cref{generalization-Ishioka-Kunugi}.

    In \cite{Kessar-Kunugi-Mitsuhashi-2011},
    the notion of Brauer indecomposability was introduced.
    If finite groups $G$ and $G'$ have a common Sylow $p$-subgroup $P$,
    and $M=S(G\times G',\Delta P)$,
    then in order to apply Brou\'{e}'s method
    \cite[Theorem 6.3]{broue-1994},
    $M(\Delta Q)$ must be indecomposable
    as a $B_{0}(C_{G}(Q))$-$B_{0}(C_{G'}(Q))$-bimodule
    for any nontrivial subgroup $Q$ of $P$.
    This means that $M$ must be Brauer indecomposable.
    On the other hand,
    if $P$ has a subgroup $Z$ central in $G$ and $G'$,
    then in order to apply \cref{generalization-Broue-gluing},
    $M(\Delta Q)$ must be indecomposable
    as a $B_{0}(C_{G}(Q))$-$B_{0}(C_{G'}(Q))$-bimodule
    only for any nontrivial subgroup $Q$ of $P$ properly containing $Z$.
    Hence we need not know $M$ to be Brauer indecomposable.
    Therefore we define relative Brauer indecomposability:

    \begin{definition}\label{def-relative-brauer-indecomposable}
        Let $M$ be a $kG$-module and $R$ a $p$-subgroup of $G$.
        We say that $M$ is relatively $R$-Brauer indecomposable
        if for any $p$-subgroup $Q$ of $G$ containing $R$,
        the Brauer construction $M(Q)$ is indecomposable $($or zero$)$
        as a $kQC_{G}(Q)$-module.
    \end{definition}

    \begin{remark}\label{remark-Brauer-construction-conjugate}
        Let $M$ be a $kG$-module.
        For any $g\in G$, we have that
        \begin{align*}
            (M(Q){\downarrow}_{QC_{G}(Q)}^{N_{G}(Q)})^{g}
            \cong
            M(Q)^{g}{\downarrow}_{(QC_{G}(Q))^{g}}^{N_{G}(Q)^{g}}
            \cong
            M(Q^{g}){\downarrow}_{Q^{g}C_{G}(Q^{g})}^{N_{G}(Q^{g})}.
        \end{align*}
        Hence $M(Q){\downarrow}_{QC_{G}(Q)}^{N_{G}(Q)}$ is indecomposable
        if and only if $M(Q^{g}){\downarrow}_{Q^{g}C_{G}(Q^{g})}^{N_{G}(Q^{g})}$
        is indecomposable.
    \end{remark}

    Let $M$ be a $kG$-module and $R$ a $p$-subgroup of $G$.
    It follows from the definition that if $M$ is Brauer indecomposable then
    $M$ is relatively $R$-Brauer indecomposable for any $p$-subgroup $R$ of $G$.
    In particular, the relative $1$-Brauer indecomposability is just
    the Brauer indecomposability.
    Moreover, if $R'$ is a $p$-subgroup of $G$ conjugate with $R$,
    then by \cref{remark-Brauer-construction-conjugate},
    $M$ is relatively $R$-Brauer indecomposable
    if and only if $M$ is relatively $R'$-Brauer indecomposable.

    We restate the definition of relative Brauer indecomposability
    for indecomposable $kG$-modules:
    \begin{lemma}\label{relative-Brauer-construction-indecomposable-characterization}
        Let $M$ be an indecomposable $kG$-module with vertex $P$,
        and $R$ a $p$-subgroup of $G$.
        Then the following are equivalent:
        \begin{enumerate}[label={$(\mathrm{\roman*})$}]
            \item The module $M$ is relatively $R$-Brauer indecomposable.
            \item The Brauer construction $M(Q)$ is indecomposable
                as a $kQC_{G}(Q)$-module or zero
                for any subgroup $Q$ of $P$ with $R\leq_{G}Q$.
        \end{enumerate}
    \end{lemma}
    \begin{proof}
        $(\mathrm{i})\Rightarrow(\mathrm{ii})$:
        This is immediate
        by \cref{remark-Brauer-construction-conjugate}.
        
        $(\mathrm{ii})\Rightarrow(\mathrm{i})$:
        Let $Q$ be any $p$-subgroup of $G$ containing $R$.
        By \cref{Brauer-construction-vertex},
        if $Q\nleq_{G}P$, then $M(Q)=0$.
        Hence we may assume that $Q\leq_{G}P$.
        Then we have that $R^{g}\leq Q^{g}\leq P$ for some $g\in G$,
        and hence that $R\leq_{G}Q^{g}\leq P$.
        Since $M(Q^{g}){\downarrow}_{Q^{g}C_{G}(Q^{g})}$ is indecomposable
        or zero, so is $M(Q){\downarrow}_{QC_{G}(Q)}$
        by \cref{remark-Brauer-construction-conjugate}.

    \end{proof}    
    
    We use the following lemmas
    in the proof of \cref{generalization-Ishioka-Kunugi}.
    \begin{lemma}\label{on-vertices-of-brauer-construction}
        Let $M$ be a trivial source $kG$-module with vertex $P$.
        If $Q$ is a $p$-subgroup of $G$ such that $Q<_{G}P$,
        then any indecomposable summand of $M(Q)$ has a vertex $R$
        such that $Q\vartriangleleft R\leq_{N_{G}(Q)}N_{P^{t}}(Q)$ for some $t\in G$.
        In particular, if $Q$ is fully $\mathcal{F}_{P}(G)$-normalized subgroup of $P$,
        then it follows that $Q\vartriangleleft R\leq_{N_{G}(Q)}N_{P}(Q)$.
    \end{lemma}
    \begin{proof}
        Let $Q$ be a $p$-subgroup of $G$ such that $Q<_{G}P$, and
        $N$ an indecomposable summand of $M(Q)$.
        Then it follows from \cite[Exercise 27.4 (b)]{Thevenaz-book-1995} that
        \[
            N\mid M(Q)
            \mid M{\downarrow}_{N_{G}(Q)}^{G}
            \mid k_{P}{\uparrow}^{G}{\downarrow}_{N_{G}(Q)}
            =\bigoplus_{t\in P\backslash G/N_{G}(Q)}
                k_{P^{t}\cap N_{G}(Q)}{\uparrow}^{N_{G}(Q)}.
        \]
        Hence we see, using \cite[Exercise 27.4 (a)]{Thevenaz-book-1995}, that
        $N$ has a vertex $R$ such that $Q\leq R\leq_{N_{G}(Q)}P^{t}\cap N_{G}(Q)$
        for some $t\in G$.
        This means that $Q$ is a normal subgroup of $R$.
        Assume that $R=Q$.
        Then the Burry-Carlson-Puig theorem
        (\cite[Theorem 4.4.6 (ii)]{Nagao-Tsushima-book-1989})
        implies that $M$ has vertex $Q$,
        which contradicts the assumption that $Q<_{G}P$.
        Hence we have that $Q\vartriangleleft R$.

        Suppose that $Q$ is fully $\mathcal{F}_{P}(G)$-normalized
        subgroup of $P$.
        Then it follows from \cref{important-properties-fusion-system} (ii)
        that $Q$ is fully automized and receptive.
        Also we have that $Q<R\leq P^{tu}$ for some $u\in N_{G}(Q)$.
        Hence we have, using \cite[Lemma 3.2]{Ishioka-Kunugi-2017}, that
        \[
            Q \vartriangleleft R\leq P^{tu}\cap N_{G}(Q)=N_{P^{tu}}(Q)\leq_{N_{G}(Q)} N_{P}(Q).
        \]
    \end{proof}

    \begin{lemma}\textnormal{(see the proof of \cite[Theorem 1.3]{Ishioka-Kunugi-2017})}
        \label{Scott-of-normalizer-is-direct-summand-of-Brauer-construction}
        Let $P$ be a $p$-subgroup of $G$ such that $\mathcal{F}_{P}(G)$ is saturated.
        If $Q$ is a fully normalized subgroup of $P$,
        then $S(N_{G}(Q),N_{P}(Q))$ is a direct summand of $S(G,P)(Q)$.
    \end{lemma}

    We now prove \cref{generalization-Ishioka-Kunugi}.
    Although the proof is essentially the same
    as that of \cite[Theorem 1.3]{Ishioka-Kunugi-2017},
    we show it for the convenience of the reader.
    
    \begin{proof}[Proof of \cref{generalization-Ishioka-Kunugi}]
        $(\mathrm{i})\Rightarrow(\mathrm{ii})$:
        This is immediate by
        \cref{Scott-of-normalizer-is-direct-summand-of-Brauer-construction}
        
        $(\mathrm{ii})\Rightarrow(\mathrm{i})$:
        By \cref{relative-Brauer-construction-indecomposable-characterization},
        it suffices to show that $M(Q){\downarrow}_{QC_{G}(Q)}$ is indecomposable
        for any subgroup $Q$ of $P$ with $R\leq_{G}Q$,
        and we show this by induction on the index $|P:Q|$.
        First, suppose that $Q=P$.
        It follows from \cite[Lemma 4.3]{Kessar-Kunugi-Mitsuhashi-2011} that
        $M(Q){\downarrow}_{QC_{G}(Q)}^{N_{G}(Q)}$ is indecomposable.
        Next, we consider the case $|P:Q|>1$,
        and assume that for any subgroup $Q'$ of $P$ with $R\leq_{G}Q'$,
        if $|P:Q'|<|P:Q|$,
        then $M(Q'){\downarrow}_{Q'C_{G}(Q')}^{N_{G}(Q')}$ is indecomposable.
        By \cref{remark-Brauer-construction-conjugate},
        we my consider without loss of generality that $Q$ is fully normalized.
        Hence it follows that $S(N_{G}(Q),N_{P}(Q))\mid M(Q)$
        by \cref{Scott-of-normalizer-is-direct-summand-of-Brauer-construction}.
        We write
        \[
            M(Q)=\bigoplus_{i=1}^{r}N_{i}  
        \]
        where $N_{1}=S(N_{G}(Q),N_{P}(Q))$
        and each $N_{i}$ is an indecomposable $kN_G(Q)$-module for $2\leq i \leq r$.
        Now suppose that $r\geq 2$ and $i$ is an integer so that $2\leq i\leq r$.
        Since $Q$ is fully normalized subgroup of $P$,
        it follows from \cref{on-vertices-of-brauer-construction} that
        $N_{i}$ has a vertex $Q'$ such that
        \[
            R\leq_{G} Q \vartriangleleft Q'\leq_{N_{G}(Q)} N_{P}(Q).
        \]
        Applying the Brauer construction with respect to $Q'$,
        we have that
        \[
            N_{1}(Q') \oplus N_{i}(Q')
            \mid
            (M{\downarrow}_{N_{G}(Q)}^{G})(Q')
            \cong
            M(Q'){\downarrow}_{N_{N_{G}(Q)}(Q')}^{N_{G}(Q)},
        \]
        where the isomorphism holds
        by the definition of the Brauer construction.
        Since $Q$ is a normal subgroup of $Q'$,
        it follows that $Q'C_{G}(Q')\leq N_{G}(Q)$.
        Hence we have the further restriction
        from $N_{N_{G}(Q)}(Q')=N_{G}(Q')\cap N_{G}(Q)$ to $Q'C_{G}(Q')$:
        \[
            N_{1}(Q'){\downarrow}_{Q'C_{G}(Q')}^{N_{G}(Q')\cap N_{G}(Q)}
            \oplus
            N_{i}(Q'){\downarrow}_{Q'C_{G}(Q')}^{N_{G}(Q')\cap N_{G}(Q)}
            \mid
            M(Q'){\downarrow}_{Q'C_{G}(Q')}^{N_{G}(Q')}.
        \]
        Since we have that $R\leq_{G}Q\vartriangleleft{Q'}^{u}\leq P$
        for some $u\in N_{G}(Q)$,
        it follows from the induction hypothesis that
        $M({Q'}^{u}){\downarrow}_{{Q'}^{u}C_{G}({Q'}^{u})}^{N_{G}({Q'}^{u})}$
        is indecomposable.
        Hence $M(Q'){\downarrow}_{Q'C_{G}(Q')}^{N_{G}(Q')}$ is indecomposable.
        However we see that $N_{1}(Q')\neq 0$ and $N_{i}(Q')\neq 0$
        since $N_{1}$ and $N_{i}$ have vertices
        containing a conjugate of $Q'$ in ${N_{G}(Q)}$.
        This is a contradiction.
        Therefore we see that $r=1$.
        Finally, by our hypothesis (ii), we have that
        \[
            M(Q){\downarrow}_{QC_{G}(Q)}^{N_{G}(Q)}
            \cong
            S(N_{G}(Q),N_{P}(Q)){\downarrow}_{QC_{G}(Q)}^{N_{G}(Q)}
        \]
        is indecomposable.
        This implies (i).

        The last assertion in the theorem has already been shown to hold
        in the argument above.
    \end{proof}

    \section{Example}

    Let $k$ be an algebraically closed field of characteristic $2$.
    Let $G=SL_{2}(11)$ and $G'=SL_{2}(3)$,
    and $B=B_{0}(G)$ and $B'=B_{0}(G')$.
    In this section, we show that $B$ and $B'$ are Morita equivalent
    by using our main theorems.

    First, we consider some subgroups of $G$ and $G'$.
    We see that $G$ and $G'$ have a common Sylow $2$-subgroup $P\cong Q_{8}$,
    the quaternion group of order $8$,
    such that $\mathcal{F}_{P}(G)=\mathcal{F}_{P}(G')$.
    We see that $Z(G)\cap P=Z(G')\cap P=:Z\cong C_{2}$.
    Let $Q_{1}$ be a cyclic subgroup of $P$ of order $4$.
    Then $P$, $Q_{1}$, and $Z$ are representatives
    of the $\mathcal{F}_{P}(G)$-conjugacy classes of
    the nontrivial subgroups of $P$.
    The centralizers of this subgroups in $G$ and $G'$ are the following:
    \begin{align}
        \label{ex-centralizers-in-G}
        C_{G}(P)=Z,\;
        C_{G}(Q_{1})\cong C_{12},\;
        C_{G}(Z)=G,\\
        \label{ex-centralizers-in-G'}
        C_{G'}(P)=Z,\;
        C_{G'}(Q_{1})\cong C_{4},\;
        C_{G'}(Z)=G'.
    \end{align}
    Note that any subgroup of $P$
    that is $\mathcal{F}_{P}(G)$-conjugate to $Q_{1}$
    is fully normalized
    since it is a normal subgroup of $P$.

    Let $M=S(G\times G',\Delta P)$.
    We show the following:
    \begin{lemma}\label{ex-Scott-delta-relative-Brauer-indecomposable}
        The Scott module $M$ is relatively $\Delta Z$-Brauer indecomposable.
    \end{lemma}
    Note that, in the example, \cref{ex-Scott-delta-relative-Brauer-indecomposable}
    immediately implies that $M$ is Brauer indecomposable
    since all the subgroups of $P$ contain $Z$
    except for the trivial subgroup.
    \begin{proof}[Proof of \cref{ex-Scott-delta-relative-Brauer-indecomposable}]
        We show by using \cref{generalization-Ishioka-Kunugi}.
        For any subgroup $Q$ of $P$,
        let
        \[
            S_{Q}=S(N_{G\times G'}(\Delta Q),N_{\Delta P}(\Delta Q))
            {\downarrow}_{(\Delta Q)C_{G\times G'}(\Delta Q)}^{}.
        \]
        We see immediately that $S_Q$ is indecomposable for $Q\in \{Z,P\}$.
        Indeed, if $Q=Z$ , then $N_{G\times G'}(\Delta Z)
        =(\Delta Z)C_{G\times G'}(\Delta Z)=G\times G'$
        and $N_{\Delta P}(\Delta Z)=\Delta P$,
        and hence $S_{Q}=\linebreak S(G\times G',\Delta P)$ is indecomposable.
        Consider the case $Q=P$.
        Then $S_{P}$ is a direct summand of
        $M(\Delta P){\downarrow}_{(\Delta P)C_{G\times G'}(\Delta P)}^{}$
        by \cref{Scott-of-normalizer-is-direct-summand-of-Brauer-construction},
        and $M(\Delta P){\downarrow}_{(\Delta P)C_{G\times G'}(\Delta P)}^{}$
        is indecomposable
        by \cite[Lemma 4.3 (ii)]{Kessar-Kunugi-Mitsuhashi-2011} as
        $\mathcal{F}_{\Delta P}(G\times G')
        \cong\mathcal{F}_{P}(G)=\mathcal{F}_{P}(G')$
        is saturated.
        Hence $S_{P}$ is indecomposable.
        
        Next, we consider the case $Q=Q_{1}$.
        We see from (\ref{ex-centralizers-in-G}) and (\ref{ex-centralizers-in-G'})
        that $C_{G\times G'}(\Delta Q)=\linebreak C_{G}(Q)\times C_{G'}(Q)$
        is $2$-nilpotent.
        This implies that $N_{G\times G'}(\Delta Q)$ is $2$-nilpotent
        as
        \linebreak$N_{G\times G'}(\Delta Q)/C_{G\times G'}(\Delta Q)
        \cong N_{G}(Q)/C_{G}(Q)
        \cong C_{2}$.
        Hence it follows from \cite[Theorem 1.4]{Ishioka-Kunugi-2017} that
        $S_{Q}$
        is indecomposable.
    \end{proof}
    
    If $g\in G$ with $Q_{1}^{g}\leq P$,
    then we see from (\ref{ex-centralizers-in-G}), (\ref{ex-centralizers-in-G'}),
    and the assumption that $\mathcal{F}_{P}(G)=\mathcal{F}_{P}(G')$
    that $Q_{1}^{g}$ is a common Sylow $2$-subgroup of
    $C_{G}(Q_{1}^{g})$ and $C_{G'}(Q_{1}^{g})$.
    Hence $C_{G}(Q)$ and $C_{G'}(Q)$ have a common Sylow $2$-subgroup
    for any subgroup $Q$ of $P$ properly containing $Z$.
    We show the following:
    \begin{lemma}\label{ex-Morita-eqiv-for-centralizers}
        Let $Q$ be any subgroup of $P$ properly containing $Z$,
        and $P_{Q}$ a common Sylow $2$-subgroup of $C_{G}(Q)$ and $C_{G'}(Q)$.
        Then the pair of $S(C_{G\times G'}(\Delta Q),\Delta P_{Q})$ and its dual
        induces a Morita equivalence between $B_{0}(C_{G}(Q))$ and $B_{0}(C_{G'}(Q))$.
    \end{lemma}
    \begin{proof}
        Since $C_{G}(Q)$ and $C_{G'}(Q)$ are $2$-nilpotent
        for any subgroup $Q$ of $P$ properly containing $Z$,
        the result follows
        (see \cite[Lemma 3.1]{Koshitani-Lassueur-2020}).
    \end{proof}

    Next, we show the following:
    \begin{lemma}\label{ex-relative-stable-equivalence}
        The pair $(M,M^{*})$ induces a relative $Z$-stable equivalence
        of Morita type between $B$ and $B'$.
    \end{lemma}
    \begin{proof}
        We show by using \cref{generalization-Broue-gluing}.
        By \cref{ex-Morita-eqiv-for-centralizers},
        it suffices to show that
        \[
            S(C_{G\times G'}(\Delta Q),\Delta P_{Q})
            \cong M(\Delta Q){\downarrow}_{C_{G\times G'}(\Delta Q)}
        \]
        for any subgroup $Q$ of $P$ properly containing $Z$,
        where $P_{Q}$ is a common Sylow $2$-subgroup of $C_{G}(Q)$ and $C_{G'}(Q)$.
        Note that if $Q$ is fully normalized,
        then it follows from \cref{generalization-Ishioka-Kunugi}
        and \cref{ex-Scott-delta-relative-Brauer-indecomposable}
        that
        \[
            S(N_{G\times G'}(\Delta Q),N_{\Delta P}(\Delta Q))
            {\downarrow}_{C_{G\times G'}(\Delta Q)}
            \cong M(\Delta Q){\downarrow}_{C_{G\times G'}(\Delta Q)}
        \]
        is indecomposable.
        
        First, we consider the case $Q=P$.
        Then we have that $C_{G}(P)=Z$ and $C_{G'}(P)=Z$,
        and that
        \[
            S(N_{G\times G'}(\Delta P),\Delta P){\downarrow}_{Z\times Z}^{}
            \mid
            k_{\Delta P}{\uparrow}^{N_{G\times G'}(\Delta P)}{\downarrow}_{Z\times Z}^{}
            \cong \bigoplus_{t\in[\Delta P\backslash N_{G\times G'}(\Delta P)/Z\times Z]}
            k_{\Delta Z}{\uparrow}^{Z\times Z}.
        \]
        This implies that
        $S(N_{G\times G'}(\Delta P),\Delta P){\downarrow}_{Z\times Z}^{}
        \cong S(Z\times Z,\Delta Z)$
        as $k_{\Delta Z}{\uparrow}^{Z\times Z}=S(Z\times Z,\Delta Z)$.
        Hence we have
        $S(Z\times Z,\Delta Z)\cong M(\Delta P){\downarrow}_{Z\times Z}^{}$.
        
        Next, let $Q$ be any subgroup of $P$
        that is $\mathcal{F}_{P}(G)$-conjugate to $Q_{1}$.
        Since
        \linebreak$|\Delta P\backslash N_{G\times G'}(\Delta Q)/C_{G\times G'}(\Delta Q)|=1$,
        we see that
        \[
            S(N_{G\times G'}(\Delta Q),N_{\Delta P}(\Delta Q))
            {\downarrow}_{C_{G\times G'}(\Delta Q)}
            \mid k_{\Delta P}{\uparrow}^{N_{G\times G'}(\Delta Q)}{\downarrow}_{C_{G\times G'}(\Delta Q)}
            \cong k_{\Delta Q}{\uparrow}^{C_{G\times G'}(\Delta Q)}.
        \]
        Hence as in the case $Q=P$, we see that
        $S(C_{G\times G'}(\Delta Q),\Delta Q)
        \cong M(\Delta Q){\downarrow}_{C_{G\times G'}(\Delta Q)}^{}$.
    \end{proof}

    We describe the structures of the restrictions to $P$
    of the nontrivial simple $B$-modules.
    The principal block $B$ has three simple modules $k_{G}$, $S_{1}$, $S_{2}$,
    where $\mathrm{dim}\,S_{i}=5$, $i=1,2$
    (see \cite[Section 9.4.4]{Bonnafe-book-2011}).
    We show the following:
    \begin{lemma}\label{ex-structures-restrictions-simple-modules}
        For $i=1$, $2$,
        \[
            S_{i}{\downarrow}_{P}^{}=k_{P}\oplus V_{i},
        \]
        where $V_{i}$ is an indecomposable $kP$-module with vertex $Z$.
    \end{lemma}
    \begin{proof}
        By \cite[Theorem 3]{Erdmann-1977},
        each $S_{i}$, as a $k[G/Z]$-module, is a trivial source module
        with vertex $P/Z$ lying in $B_{0}(G/Z)$,
        and hence is a trivial source module
        with vertex $P$ lying in $B$.
        By \cite[Table 9.1]{Bonnafe-book-2011},
        $S_{1}$ and $S_{2}$ afford the irreducible characters
        ${R'}_{+}(\theta_{0})$ and ${R'}_{-}(\theta_{0})$, respectively.
        Let $\chi_{i}$, $i=1,\ldots,4$,
        be the linear characters of $Q_{8}$,
        where $\chi_{1}$ is the trivial character.
        We see from the character table of $Q_{8}$ that 
        ${R'}_{+}(\theta_{0}){\downarrow}_{P}
        ={R'}_{-}(\theta_{0}){\downarrow}_{P}
        =2\chi_{1}+\chi_{2}+\chi_{3}+\chi_{4}$.
        Hence if we write $S_{i}{\downarrow}_{P}=k_{P}\oplus V_{i}$
        for some $kP$-module $V_{i}$, $i=1$, $2$,
        then each $V_{i}$ affords the same character
        $\varphi=\chi_{1}+\chi_{2}+\chi_{3}+\chi_{4}$.
        
        We show that each $V_{i}$ is indecomposable.
        Suppose that $V_{i}$ is not indecomposable.
        Then there is a nontrivial partition
        $\{\chi_{j}\mid j=1,\ldots,4\}=I_{1}\sqcup\cdots\sqcup I_{\ell}$,
        $2\leq \ell\leq 4$,
        such that the characters $\sum_{\chi\in I_{m}}\chi$,
        $m=1,\ldots,\ell$,
        are precisely the characters
        afforded by indecomposable summands of $V_{i}$.
        Moreover, by \cite[Chapter II, Lemma 12.6 (ii)]{Landrock-book-1983},
        each $\sum_{\chi\in I_{m}}\chi$ must take nonnegative integers.
        However, we see from the character table of $Q_{8}$ that
        for any nontrivial partition
        $\{\chi_{j}\mid j=1,\ldots,4\}={I'_{1}}\sqcup\cdots\sqcup {I'_{\ell'}}$,
        there exists $m$ such that $\sum_{\chi\in I'_{m}}\chi$
        takes negative values, a contradiction.
        
        We see that $\varphi$ takes zero
        except for the elements of $Z$.
        Hence it follows from \cite[Chapter II, Lemma 12.6 (iii)]{Landrock-book-1983}
        that $Z$ is a vertex of each $V_{i}$.
    \end{proof}
    
    Finally, we show that $(M,M^{*})$ induces a Morita equivalence between $B$ and $B'$.
    By \cref{ex-relative-stable-equivalence} and \cref{generalization-Linckelmann} (iii),
    it suffices to show that the functor $-\otimes_{B}M$ sends the simple $B$-modules
    to simple $B'$-modules.
    It follows from
    \cref{generalization-Linckelmann} (ii)
    and \cref{equivalent-condition-having-Scott-as-direct-summand}
    that
    $k_{G}\otimes_{B}M\cong k_{G'}$.
    Since $M$ is a direct summand of $k_{\Delta P}{\uparrow}^{G\times G'}$,
    using \cref{ex-structures-restrictions-simple-modules},
    we have that for $i=1$, $2$,
    \[
        S_{i}\otimes_{B}M
        \mid S_{i}{\downarrow}_{P}{\uparrow}^{G'}
        \cong k_{P}{\uparrow}^{G'}\oplus V_{i}{\uparrow}^{G'}.
    \]
    Since $V_{i}{\uparrow}^{G'}$ is $Z$-projective,
    it follows from \cref{generalization-Linckelmann} (ii)
    that $S_{i}\otimes_{B}M$ is an indecomposable summand of $k_{P}{\uparrow}^{G'}$.
    We see that $k_{P}{\uparrow}^{G'}\cong k_{G'}\oplus T_{1}\oplus T_{2}$,
    where $T_{1}$ and $T_{2}$ are $1$-dimensional nontrivial simple $B'$-modules,
    and hence $S_{i}\otimes_{B}M$ is simple.
    Thus by \cref{generalization-Linckelmann} (iii),
    $(M,M^{*})$ induces a Morita equivalence between $B$ and $B'$.
  
    \section*{Acknowledgement}
    This work was supported by JSPS KAKENHI Grant Number JP18K03255.

    \vspace*{\baselineskip}

    \noindent Naoko Kunugi,
    Department of Mathematics,\\
    Tokyo University of Science,
    1-3 Kagurazaka, Shinjuku-ku, Tokyo, 162-8601, Japan\\
    E-mail: kunugi@rs.tus.ac.jp

    \vspace*{\baselineskip}

    \noindent
    Corresponding author:\\
    Kyoichi Suzuki,
    Department of Mathematics,
    Graduate School of Science,\\
    Tokyo University of Science,
    1-3 Kagurazaka, Shinjuku-ku, Tokyo, 162-8601, Japan\\
    E-mail: 1119703@ed.tus.ac.jp

\end{document}